\documentclass{amsart}[12 pt]
\usepackage{amsfonts}
\usepackage{amsmath}
\usepackage{amsthm}
\usepackage{amssymb}
\usepackage[pdftex]{hyperref}
\usepackage{url}
\newtheorem{thm}{Theorem}[section]
\newtheorem{defn}{Definition}[section]
\newtheorem{lemma}{Lemma}[section]
\newtheorem*{ack}{Acknowledgements}
\newtheorem{prop}{Proposition}[section]

\newtheorem{cor}{Corollary}[section]
\newtheorem{remark}{Remark}[section]
\newcommand{\bb}{\mathbb}
\newcommand{\delbar}{\bar{\partial}}

\title{Contact homology of $S^1$-bundles over some symplectically reduced orbifolds}
\author{Justin Pati}
\address{Department of Mathematics and Statistics\\ 
MSC03 2150 1 University of New Mexico\\ 
Albuquerque, New Mexico\\
87131-0001}
\email{jpati@math.unm.edu}
\begin{document}

\begin{abstract}
In this paper, we compute contact homology of some quasi-regular contact structures, which admit Hamiltonian actions of Reeb type of Lie groups.  We will discuss the toric contact case, (where the torus is of Reeb type), and the case of homogeneous contact manifolds.
In both of these cases the quotients by the Reeb action are K\"{a}hler and admit perfect Morse-Bott functions via the moment map.  It turns out that the contact homology depends only on the homology of the symplectic base and the bundle data of the contact manifolds as a circle bundle over the base.  Moreover we can identify the relevant holomorphic spheres in the bases which lift to gradient trajectories of the action functional on the loop space of the symplectization of M.  In order to include all toric contact manifolds we extend a result of Bourgeois to the case of circle bundles over symplectic orbifolds.  In particular we are able to distinguish contact structures on many contact manifolds all in the same first Chern class of $2n$-dimensional plane distributions.    
\end{abstract}

\maketitle

\section{Introduction}
Like symplectic manifolds, contact manifolds have no local invariants.  Darboux's theorem tells us that locally all contact structures are the same. Moreover Gray stability tells us that there is no deformation theory.  Nonetheless there are contact manifolds which are not contactomorphic.  
Sometimes one can distinguish these different contact structure via the Chern classes of the underlying symplectic vector bundle defined by the contact distribution.  However this is insufficient.  In ~\cite{Gi1} Giroux shows that the contact structures $$\xi_{n} = ker(\alpha_{n} =cos(n\theta )dx + sin(n\theta)dy)$$ are pairwise noncontactomorphic.  Moreover $c_{1}(\xi_{n}) = 0$ for all $n$.

Calculations of this type have been entirely dependent on the specific geometric conditions of the example.  However, we now have contact homology, a small piece of the symplectic field theory ~\cite{EGH} of Elisahberg, Givental, and Hofer.  Using this powerful tool Ustilovsky ~\cite{Ust} was able to find infinitely many exotic contact structures on odd dimensional spheres all in the same homotopy class of almost complex structures.  Similarly Otto Van Koert, in his thesis ~\cite{OVK}, made a similar calculation for a larger class of Brieskorn manifolds using the Morse-Bott form of the theory.  It should be mentioned that it came to the author's  attention upon completion of this work that Miguel Abreu, using different methods, has, indepentently, computed a general formula for contact homology for toric contact manifolds with $c_{1}(\xi)=0.$    

The ideas in this paper were originally motivated by examples related to a question of Lerman about contact structures on various $S^{1}$-bundles over $\bb{CP}^1 \times \bb{CP}^1$ ~\cite{Ler1}.  We are now able to distinguish these structures essentially using an extension of a theorem of Bourgeois for $S^1$-bundles over symplectic manifolds which admit perfect Morse functions ~\cite{Bourg1}.  
\begin{thm}[Bourgeois]
Let $(M, \omega)$ be a symplectic manifold with $[\omega] \in H_2 (M, \bb{Z})$  satisfying $c_1(TM) =\tau[\omega]$ for some $\tau \in \bb{R}.$ Assume that $M$ admits a perfect Morse function. Let $V$ be a Boothby-Wang fibration over $M$ with its natural contact structure.  Then contact homology $HC^{cylindrical}_{∗} (V, \xi)$ is the homology of the chain complex generated by infinitely many copies of $H_{*}(M, \bb{R})$, with degree shifts $2 c k - 2$, $k \in \bb{N},$ where $c$ is the first Chern class of $T(M)$ evaluated on a particular homology class. The differential is then given in terms of the Gromov-Witten potential of $M.$
\end{thm}

This theorem exploits the fact that in the case of $S^{1}$-bundles the differential in contact homology is especially simple since there is essentially one type of orbit for each multiplicity (ie., simple orbits can be paramatrized so that their periods are all $1$).  Perfection of the Morse function kills the Morse-Smale-Witten part of the differential in Morse-Bott contact homology, so the chain complex reduces to the homology of $M.$  The grading in contact homology comes from the fact that the index calculations may be made via integration of $c_{1}(TM)$ over certain spherical two dimensional homology classes.  In the case of simply connected reduced spaces the cohomology ring of the base has a particularly nice form in terms of $2$ dimensional cohomology classes, obtained from the moment map as a Morse function.  Moreover all the two dimensional homology of the base in these cases is generated by spheres.  This fact makes Maslov/Conley-Zehnder indices easier to compute without the need to find a stable trivialization of the contact distribution.  In this way we are able to extend the theorem of Bourgeois.  For homogeneous spaces the above theorem works ``out of the box'', but we must allow for non-monotone bases, ie., $c_{1}(\xi) \neq 0$.  To generalize the idea to general toric contact manifolds it is necessary to consider symplectic bases (of the circle bundle) which are orbifolds.\footnote{It would be interesting to determine conditions under which the Reeb vector field could be perturbed so that the quotient of the Reeb action becomes an honest manifold when the contact manifold is simply connected.  There are known counter-examples, even in the toric case.  It would be interesting to find SFT obstructions to this.} The main new idea of this paper is to formalize this process for index computation from ~\cite{Bourg1} and extend the result to an orbifold base, ie., the case where Reeb action is only locally free.     Moreover we can extend this to orbifolds bases which also admit a Hamiltonian action, since over $\bb{C}$ their cohomology ring is still a polynomial ring in $H^{2}$ with spherical representatives of the ``diagonal'' homology classes.  
\begin{thm} \label{thm:main}
 Let $M$ be a contact manifold, which is an $S^1$ bundle over the symplectic orbifold $\mathcal{Z},$ where $\mathcal{Z}$ admits a strongly Hamiltonian action of a compact Lie group.  Suppose that the curvature form, $d \alpha$ of $M$ as a circle bundle over $\mathcal{Z}$ is given by $$\sum w_{j}\pi^{*}c_{j},$$ where $c_{j}$ are the Chern classes associated to the Hamiltonian action.  Assume that the $c_{j}$ generate $H^{*}(\mathcal{Z};\bb{C})$ as variables in a truncated polynomial ring.  Let $\tilde{w_{j}}$ be the coeffecient of $c_{j}$ in $c_{1}(T(\mathcal{Z})).$  Assume further regularity of the $\delbar_{J}$-operator, and that $\sum_{j} \tilde{w_{j}} >1.$  Then contact homology is generated by copies of the homology of the critical submanifolds for any of the Morse-Bott functions given by the components of the moment map, with degree shifts given by twice-integer multiples of the sums of the $\tilde{w_{j}}$ plus the dimension of the stratum, $S$, of $\mathcal{Z}$ in which the given Reeb orbit is projected under $\pi.$ 
\end{thm}
As corollaries we obtain contact homology for both toric and homogeneous contact manifolds.  The reader should also beware that such calculations are only formal without some sort of transversality of the $\delbar_{J}$ operator.  Though several such statements and proofs have been published there are still certain persistent gaps.  There are many current developments around this issue and the underlying analysis.  Without such a result there is no way to know if the counts that we make are actually correct.  In section $3$ we give an alternate argument for transversality for homogeneous contact manifolds and for toric contact manifolds using only elementary tools from algebraic geometry, this is possible only because the almost complex structures involved are integrable in these cases.

\begin{ack}
 I would like to thank Charles Boyer for his patience, interest, and assistance in my work.  I would also like to thank Alexandru Buium for many useful conversations about the dark arts of algebraic geometry.
\end{ack}

\section{Preliminary Notions}
In this section we give some basic definitions, theorems, and ideas in order to fix notation and perspective.  
First of all we define a contact structure:
\begin{defn}
 A \textbf{contact structure} on a manifold $M$ of dimension $2n-1$ is a \textbf{maximally non-integrable} $2n-2$-plane distribution, $\xi \subset T(M).$
In other words $\xi$ is a field of $2n-2$-planes which is the kernel of some $1$-form $\alpha$ which satisfies $$\alpha \wedge d\alpha^{n-1} \neq 0.$$  Such an $\alpha$ is called a \textbf{contact 1-form}.    
\end{defn}
Notice that given such a $1$-form $\alpha$, $f\alpha$ is also a contact $1$-form for $\xi$ whenever $f$ is smooth and non-vanishing.  In the following, a contact manifold has dimension $2n-1$, hence the symplectization has dimension $2n$ and our symplectic bases all have dimension $2n-2.$  Unless otherwise stated we assume $M$ compact without boundary. 
Given a choice of contact $1$-form, $\alpha$ we define the \emph{Reeb vector field} of
 $\alpha$ as the unique vector field $R_{\alpha}$ satisfying $$i_{R_{\alpha}} d \alpha = 0$$ and $$i_{R_{\alpha}} \alpha = 1.$$   
 A contact form $\alpha$ is called \emph{quasi-regular} if, in flow-box coordinates with respect to the flow of the Reeb vector field, orbits intersect each flow box a finite number of times.  If that number can be taken to be $1$ we call the $\alpha$ \emph{regular}.  Note that if the manifold is compact this makes all Reeb orbits periodic\footnote{By Poincar\'{e} recurrence, for example.} of the same period.  In contrast, the standard contact form on $\bb{R}^{2n-1}$ is regular.  We call a contact manifold (quasi)regular if there is a contact form $\alpha$ for $\xi$ such that $\alpha$ is (quasi)regular.  Given $(M, \alpha)$, we denote by $V$ the \emph{symplectization} of $M$: $$V := (M \times \bb{R}, \omega = d(e^t \alpha)).$$  A contact structure defines a symplectic vector bundle with transverse symplectic form $d\alpha.$  We can then choose an almost complex structure $J_{0}$ on $\xi.$  We extend this to a complex structure $J$ on $V$ by $$J|_{\xi}= J_{0}$$ and $$J\frac{\partial}{\partial t} = R_{\alpha}$$ where $t$ is the variable in the $\bb{R}$-direction.  Note that we also get a metric, as usual,  on $V$ compatible with $J$ given by $$g(v, w) = \omega(v, Jw).$$

We have the important result from ~\cite{BG00b}, originally proved in the regular case in ~\cite{BW58}:
\begin{thm}[Orbifold Boothby-Wang]
 Let $(M, \xi)$ be a quasiregular contact manifold.  Then $M$ is a principal  $S^1$-orbibundle over a symplectic orbifold $(\mathcal{Z}, \omega)$ with connection $1$-form $\alpha$ whose curvature satisfies $$d\alpha = \pi^*\omega.$$  If $\alpha$ is \textbf{regular} then $\mathcal{Z}$ is a manifold, and $M$ is the total space of a principal $S^1$-bundle. 
\end{thm}

This enables us to study the nature of $M$ via the cohomology of $\mathcal{Z}.$  As we will see, in nice enough cases, the cohomology of $\mathcal{Z}$ along with the bundle data of the Boothby-Wang fibration will determine the contact homology as well.  Notice that if we really want to study the quasi-regular case via the base we are forced to consider symplectic orbifolds, a complete study of contact geometry with symmetries necessarily must include orbifolds.

The first important examples here are the homogeneous contact manifolds.  These are contact manifolds which have a transitive action of a Lie group via contactomorphisms.  In all of these examples we can apply the theorem of Bourgeois \emph{almost} directly\footnote{We say almost, since these manifolds are not always monotone. ie., that $c_{1}$ is a multiple of the symplectic form.}.  Next we have the toric contact manifolds.  The toric situation is not quite as nice as with the homogeneous manifolds, but we will still be able to make a similar statement and read off contact homology from the Boothby-Wang data.
Before we describe the contact situation in detail, we describe the corresponding symplectic objects over which we will construct our $S^1$ bundles.  We will have to work with symplectic orbifolds as in some of our examples the action of the Reeb vector field is only locally free.
\subsection{Toric Symplectic and Toric Contact Manifolds}

\begin{defn}
 Let $M$ be a symplectic orbifold.  Suppose that the Lie group $G$ acts on $M$ by (orbifold) symplectomorphisms.  The action is called \textbf{Hamiltonian} is for each $\tau \in \mathfrak{g}$ the associated vector field $X_{\tau}$ on $M$ satisfies $$i_{X_{\tau}}\omega = dH$$ for some $H \in C^{\infty}(M , \bb{R}).$
\end{defn} 
We will actually be interested in something stronger.  We will assume that our actions are what some authors call \emph{strongly} Hamiltonian, ie., that there exists a \emph{moment map} $$\mu : M \rightarrow \mathfrak{g}^*$$ such that 
 $$d\langle \mu(p),\tau \rangle=-i_{X_{\tau}} \omega.$$

This result ~\cite{LerTol} will be extremely important for us when we determine how to compute contact homology.
\begin{thm}(Lerman-Tolman)
 Let $(M, \omega)$ be a symplectic orbifold.  Let $G$ be a compact Lie group acting effectively, and such that the action is strongly Hamiltonian with moment map $\mu$.  Then each component of the moment map is a Morse-Bott function whose critical submanifolds are all even dimensional with even Morse-Bott index.  
\end{thm}

When the group $G$ is a torus of maximal possible dimension we have a name for these manifolds:
\begin{defn} Suppose that a symplectic orbifold $(M, \omega)$ of dimension $2n$ admits an effective Hamiltonian action of a torus $T^n$ of dimension $n.$ Then we call $(M, \omega)$ a \textbf{toric} symplectic orbifold.
 \end{defn}

Now we consider a compact Lie group $G$ acting via contactomorphisms on the contact manifold $(M^{2n-1}, \xi).$   By averaging over $G$ we can always obtain a $G$-invariant contact form for $\xi.$  We find it convenient to limit all of our actions.  First given $\zeta \in \mathfrak{g}$, let $X_{\zeta}$ denote the fundamental vector field associated to $\zeta$ defined via the exponential map.  The following definition was introduced in ~\cite{BG00}.  

\begin{defn}  Let $G$ be a Lie group which acts on the contact manifold $(M, \xi).$  The action is said to be of \textbf{Reeb type} if there is a contact $1$-form $\alpha$ for $\xi$ and an element $\zeta \in \mathfrak{g}$, such that $X^{\zeta}=R_{\alpha}.$  
\end{defn}

 This is important since as long as such an action is proper, we know that our contact manifold is an $S^1$-bundle.  The proof of the following can be found in ~\cite{BG} 

\begin{prop}
 If an action of a torus $T^k$ is of Reeb type then there is a quasi-regular contact structure whose $1$-form satisfies, $ker(\alpha)=\xi.$ Moreover, then $M$ is the total space of a $S^1$ bundle over a symplectic orbifold which admits a Hamiltonian action of a torus $T^{k-1}.$  
\end{prop}

We now need the definition:
\begin{defn}
 A \textbf{toric contact manifold} is a co-oriented contact manifold $(M^{2n-1}, \xi)$ with an effective action of a torus, $T^{n}$ of maximal dimension $n$ and a moment map\footnote{The contact moment map can be defined in terms of the symplectization, $V$, or intrinsically in terms of the annhilator of $\xi$ in $TM^*$.  For more information about this see ~\cite{Ler2} or ~\cite{BG}}  into the dual of the Lie algebra of the torus.  
\end{defn}

Notice now that by constraining to actions of Reeb type we get that all of our toric contact manifolds are circle bundles over symplectic orbifolds.  Even more is true, the bases are toric.  So we limit our study strictly to those tori of Reeb type.\footnote{Actually, our homology calculations are potentially false otherwise, since in dimension $3$, there are toric manifolds with over-twisted contact structures. These always have vanishing contact homology (when contact homology is well-defined).}  An action being of Reeb type is quite a strong condition, but by the classification theorem of Lerman ~\cite{Ler2} we see that it includes many cases of interest.  We only include his classification for $2n-1 >3.$
\begin{thm}
 Let $(M, \xi)$ be a toric contact manifold of dimension greater than $3.$  
\begin{enumerate}
 \item[i.] If the torus action is free, then $M$ is a principal torus bundle over a sphere.
\item[ii.] If the action is not free then the \textbf{moment cone} is a \textbf{good rational polyhedral cone}. 
\end{enumerate}

\end{thm}
We are interested mostly in the second case and since we have the following proposition in~\cite{BG} it is completely reasonable for us to restrict ourselves to actions of Reeb type.
\begin{prop}
 Let $M$ be a contact toric manifold of dimension greater than $3$.  Then the action of the torus is of Reeb type if and only if it is not free and the moment cone contains no non-zero linear subspace. 
\end{prop}

\subsection{Cohomology Rings of Reduced Spaces}
In this section we follow ~\cite{GlSt}.  The thing that really makes our calculation possible in its simple form is the wonderful structure of the cohomology rings of symplectically reduced spaces, ie, they are all truncated polynomial rings in the Chern classes.  Moreover in the simply connected case, we know that all of $H_{2}$ can be represented by spheres.  Even better, we can always relate all of these homology and cohomology classes to the moment map.       

First let's work out what we get in general. Let $(M, \omega)$ be a symplectic manifold of dimension $2n.$ Let $G$ be a compact connected Lie group of dimension d which acts via (strongly) Hamiltonian symplectomorphisms, $\mathfrak{g} = Lie(G).$  Let $$\mu: M \rightarrow \mathfrak{g^{*}}$$ denote the corresponding moment map.  Let $\tau$ be a regular value of $\mu$, let $$X_{\tau}=\mu^{-1}(\tau).$$  Then $X_{\tau}/G$ is a symplectic orbifold of dimension $2(n-d)$.  Set $\mathcal{Z} = X_{\tau}/G$.  Let $c_{1}, \ldots, c_{n}$ be the Chern classes of the fibration $M \rightarrow \mathcal{Z}.$

\begin{thm}
 If the $c_{1}, \ldots, c_{n}$ generate $H^{*}(\mathcal{Z} ; \bb{C})$ then
$$H^{*}(\mathcal{Z}; \bb{C}) \simeq \bb{C}[x_{1}, \ldots, x_{n}]/ann(v_{top})$$
where $ann(v_{top})$ is the annhilator of the highest order homogeneous part of the symplectic volume of $\mathcal{Z}.$
\end{thm}

Again the following results are in ~\cite{GlSt}
\begin{prop}
 Let $\mathcal{Z}$ be a toric orbifold.  Then the Chern classes as above generate $H^{*}(\mathcal{Z};\bb{C}).$ 
\end{prop}
To apply this to all homogeneous contact manifolds we need not only the case of flag manifolds but also of \emph{generalized} flag manifolds.  These are quotients of a complex semi-simple Lie group $G$ by a \emph{parabolic} subgroup $P.$  These include the flag manifolds.  We extend the result from ~\cite{GlSt} about flag manifolds to $G/P.$  For more about generalized flag manifolds see ~\cite{BE} and ~\cite{BGG}
\begin{prop}
 Let $G/P$ be a generalized flag manifold.  Then the cohomology is generated by the Chern classes as above. 
\end{prop}
\begin{proof}
 Since $P$ is parabolic, it contains a Borel subgroup.  Each Schubert cell in $G/P$ lifts to one in $G/B.$  This gives an injective map $$H^{*}(G/P ; \bb{C}) \rightarrow H^{*}(G/B;\bb{C}).$$  Thus we need only to see that the Chern classes generate $H^{*}(G/B; \bb{C})$ which is known from ~\cite{Bor1}.
\end{proof}

\subsection{The Maslov Index and the Conley-Zehnder Index}
The Maslov index associates to each path of symplectic matrices a rational number. This particular definition originally appeared in Salamon and Robbins.  This index determines the grading for the chain complex in contact homology.  The Maslov index should be thought of as analagous to the Morse index for a Morse function.  The analogy isn't perfect, since the actual Morse theory we consider should give information about the loop space of the contact manifold.  Also note that our action functional has an infinite dimensional kernel.   
It should be noted that we will describe three indices in the following.  Two of them will be called the Maslov index.  This is unfortunate, but it will always be clear which Maslov index we will use at any particular time.
\begin{remark}Historically, the Maslov index arose as an invariant of loops of Lagrangian subspaces in the Grassmanian of Lagrangian subspaces of a symplectic vector space V.  In this setting the Maslov index is the intersection number of a path of Lagrangian subspaces with a certain algebraic variety called the Maslov cycle.
This is of course related to our Maslov index of a path of symplectic matrices, since we can consider a path of Lagrangian subspaces given by the path of graphs of the desired path of symplectic matrices. For more information on this see ~\cite{McDSal}, and ~\cite{SalRob}.
\end{remark}
\begin{remark}For a symplectic vector bundle, $E$,  over a Riemann surface, $\Sigma$ there is symplectic definition of the first Chern number $\langle c_{1}(E),\Sigma \rangle$.  It turns out that this Chern number is the loop Maslov index of a certain loop of symplectic matrices, obtained from local trivializations of $\Sigma$ decomposed along a curve $\gamma \subset \Sigma$.  This Chern number agrees with the usual definition, considering $E$ as a complex vector bundle, and can be obtained via a curvature calculation. 
\end{remark}
Let $\Phi(t)$, $t \in[0, T]$ be a path of symplectic matrices starting at the identity such that $det(I - \Phi(T)) \neq 0$\footnote{This is the \emph{non-degeneracy} assumption.  In the context of the Reeb vector field, this condition implies that all closed orbits are isolated}.  We call a number $t \in [0,T]$, a \emph{crossing} if $det(\Phi(t) - I)=0.$  A crossing is called \emph{regular} if the \emph{crossing form} (defined below) is non-degenerate.  One can always homotope a path of symplectic matrices to one with regular crossings, which, as we will see below, does not change the Maslov index. 

For each crossing we define the \emph{crossing form} 
$$\Gamma(t)v = \omega(v, D \dot{\Phi}(t)).$$ 
Where $\omega$ is the standard symplectic form on $\bb{R}^{2n}.$
\begin{defn}
 The Conley-Zehnder of the path $\Phi{t}$ under the above assumptions is given by:
$$\mu_{CZ}(\Phi) = \frac{1}{2}sign(\Gamma(0)) + \sum_{t\neq 0\,, \,t\, a\,\, crossing}sign(\Gamma(t))$$
\end{defn}
The Conley-Zehnder index satisfies the following axioms: 
\begin{itemize}
\item [Homotopy:] $\mu_{CZ}$ is invariant under homotopies which fix endpoints.
\item [Naturality:] $\mu_{CZ}$ is invariant under conjugation by paths in $Sp(n, \bb{R}).$
\item [Loop:] For any path, $\psi$ in $Sp(n, \bb{R}),$ 
and a loop $\phi$, $$\mu_{CZ}(\psi \cdot \phi) = \mu_{CZ}(\psi) + \mu_{l}(\phi).$$  Where $\mu_{l}$ is the Maslov index for loops of symplectic matrices.
\item [Direct Sum:]  If $n = n' + n''$ and $\psi_{1}$ is a path in $Sp(n', \bb{R})$ 
and $\psi_{2}$ is a path in $Sp(n'', \bb{R})$ 
then for the path $\psi_{1} \oplus \psi_{2} \in Sp(n', \bb{R}) \bigoplus Sp(n'', \bb{R}),$ we have 
$$\mu(\psi_{1} \oplus \psi_{2}) = \mu(\psi_{1}) + \mu(\psi_{2}).$$
\item [Zero:]  If a path has no eigenvalues on $S^1$, 
then its Conley-Zehnder index is 0.
\item [Signature:]  Let $S$ be symmetric and nondegenerate with $$||S|| < 2\pi.$$  Let $\psi(t) = exp(JSt)$, then $$\mu_{CZ}(\psi) = \frac{1}{2}sign(S).$$  
\end{itemize}
The Conley-Zehnder index is still insufficient for our purposes since we need the assumption that at time $T=1$ the symplectic matrix has no eigenvalue equal to 1.
We introduce yet another index for arbitrary paths from ~\cite{SalRob}.  We will call this index the Maslov index and denote it $\mu.$  

For this new index we simply add half of the signature of the crossing form at the terminal time of the path to the formula for the Conley-Zehnder index.
$$ \mu(\Phi(t)) = \frac{1}{2}sign(\Gamma(0)) + \sum_{t\neq 0\,, \,t\, a\,\, crossing}sign(\Gamma(t)) + \frac{1}{2}sign(\Gamma(T))$$
This Maslov index satisfies the same axioms as $\mu_{CZ}$ as well as the new property of catenation.  This means that the index of the catenation of paths is the sum of the indices.

\subsection{Indices for homotopically trivial closed Reeb orbits}
Let $\gamma$ be a closed orbit of a Reeb vector field.  Choose a symplectic trivialization of this orbit in $M,$ ie., take a map $u : D \rightarrow M$ from a disk into $M,$ with the property that the boundary of the image of $u$ is $\gamma$ and a bundle isomorphism between $u^* \xi$ and standard symplectic $\bb{R}^{2n}$, $(\bb{R}^{2n} , \omega_{0}).$   Now we look at the Poincare time $T$ return map of the associated flow (with respect to this trivialization, choosing a framing), where $T$ is the period of $\gamma.$  If the linearized flow  has no eigenvalue equal to $1$, we define the Conley-Zehnder index of $\gamma$ to be the Conley-Zehnder index of the path of matrices given by the linearized Reeb flow.  If there are eigenvalues equal to $1$ we calculate the Maslov index of the path of matrices coming from the flow (in an appropriate symplectic trivialization.)  Note that when there is no eigenvalue equal to $1$, the two indices agree.  

The Conley-Zehnder and Maslov indices depends on the choice of spanning disk or Riemann surface used in the symplectic trivialization.  Different choices of disks will change the index by twice the first Chern class\footnote{This is the reason that so often in the literature on contact homology authors insist that $c_{1}(\xi) = 0$.  This index defines the grading of contact homology so if this Chern class is non-zero we must be careful to keep track of which disks we use to cap orbits.} of $\xi$.  Intuitively, given a periodic orbit of the Reeb vector field, this index reveals how many times nearby orbits ``wrap around'' the given orbit.

\section{Morse-Bott Contact Homology}
Our invariants come from a sort of infinite dimensional Morse theory on the loop space of the symplectization $V$ of $M$.  In this context, the analogue of gradient trajectories is given by pseudo-holomorphic images of punctured Riemann surfaces.  For us a $J$-holomorphic curve is a $C^{\infty}$ function from a Riemann surface into $V$ whose total derivative is almost-complex linear with respect to $J$ on $V$, and the standard complex structure on the Riemann surface.  

We want to count so-called \emph{rigid} curves between critical points of a Morse functional.  To do this we study certain moduli spaces of \emph{stable} maps of punctured Riemann surfaces into $V.$  As we will see these curves are \emph{asymptotically cylindrical} over closed Reeb orbits.  The number of punctures corresponds to the number of orbits.  We denote such moduli spaces by $$\mathcal{M}_{0,J}^A(\gamma_{1}, \ldots ,\gamma_{s};\gamma_{s+1}, \ldots ,   \gamma_{s+k},V).$$  This reads ``the set of genus $0$ ie., \emph{rational} maps into $V$ with $s$ positive punctures and $k$ negative punctures representing the $2$ dimensional homology class $A$.  When no confusion can arise we will drop the ''$V$`` from the notation.  Also, when it is understood that there is only one positive puncture, we omit the semicolon.  
The trouble is that the relevant moduli spaces of curves have, at times, quite an elusive geometric structure.  We try to capture their structure as zero sets of a section of a certain infinite dimensional vector bundle over the space of maps from $\Sigma$ into $V.$  The relevant section is a Fredholm operator with a nice index given in terms of the dynamics of the Reeb vector field of $\alpha.$  The problem is that one cannot directly apply the implicit function theorem to determine the dimension of the moduli space since the linearized operator could have a nonzero cokernel.  Therefore, until we can prove that the operator is transverse to the zero section of the bundle we cannot conclude any kind of manifold structure or dimension formula from the Fredholm index of this operator. 

In the cases at hand we have an integrable almost complex structure so we can try to use Dolbeault cohomology of an appropriate complex.  As it turns out, this works under suitable positivity conditions, so we can determine transversality from bounds on spherical chern numbers in $M.$  

Contact homology is the homology of the differential graded algebra generated by periodic orbits of the Reeb vector field and graded by the reduced Conley-Zehnder index equal to $\mu_{CZ} +n-3.$  Whenever there are no Reeb orbits of reduced Conley-Zehnder index equal to $0$, $1$, or $-1$, we can actually use a simpler chain complex.  We then consider the graded \emph{vector space} generated by the closed Reeb orbits, in this case we call the homology \emph{cylindrical contact homology}.\footnote{There is an alternative to this called \emph{linearized contact homology} which ''chops off`` unwanted legs in pairs of pants via an augmentation.  This homology is well-defined even when cylindrical homology is not.  Whenever cylindrical homology \emph{is} well-defined these two invariants are the same.}  The reader should be wary that when the first Chern class of the contact distribution\footnote{We consider $\xi$ as a symplectic vector bundle, defining a first Chern class.} is non-zero, then in order to define a reasonable grading we must keep track of all choices made, spanning surfaces for Reeb orbits, for example.    
For cylindrical contact homology the differential is given as follows:
  $$ d\gamma=\sum_{\gamma^{'}} n_{\gamma , \gamma^{'}} \gamma^{'}.$$ Where $$n_{\gamma, \gamma^{'}} = 0$$\\*
whenever $$dim(\mathcal{M}_{0,J}^{A}(\gamma, \gamma^{'}, V))/\bb{R} \neq 0$$
otherwise it is equal to the signed count of these curves.\footnote{Recall that we have an action of $\bb{R}$ on the compact moduli space $\mathcal{M}_{0}^{A}(\gamma, \gamma^{'}, V, J)$ by translation in $t$, thus we make a count when the \emph{quotient} is $0$-dimensional.}
The following is from ~\cite{EGH}.
\begin{thm}
 Suppose that there are no Reeb orbits of Conley-Zehnder index $0, -1, 1.$  Then $d\circ d =0$, and cylindrical contact homology, $CH_{*}(M, \xi)$ is defined to be the homology of the above chain complex.  Moreover it is independent of all choices made, including $J$ and $\alpha$, in particular it is an invariant of the \textbf{contact structure} $\xi.$ 
\end{thm}

This differential counts \emph{rigid} \footnote{Rigid, in this context means that these curves live in $0$-dimensional moduli spaces} $J$-holomporphic curves in the symplectization of $M$ asymptotically over $\gamma$ at $\infty$ and over $\gamma^{'}$ at $- \infty.$ 

\begin{remark}
 Contact homology and in particular, the Floer-like cylindrical contact homology, are just the tip of the iceburg of the larger \textbf{symplectic field theory} which contains many more subtle invariants.
\end{remark}

To use this complex, however, we really need the condition that the periodic Reeb orbits are isolated.  Since we are working exclusively with $S^1$-bundles, all Reeb orbits are periodic so are never isolated.  It turns out there there is a Morse-Bott version of all of this theory.  

To make this work we actually work with Morse theory on the base.  If we work in an orbifold we have different orbifold strata given by the various orbit types of the original action.  Recall that we assume all actions are strongly Hamiltonian.
We now introduce some terminology.
\begin{defn}
 The \textbf{action functional} is the map $$\mathcal{A} : C^{\infty}(S^1, M) \rightarrow \bb{R}$$ given by $$\gamma \mapsto \int_{\gamma} \alpha,$$ where $\alpha$ is a contact 1-form for $(M, \xi).$
 \end{defn}
Note that critical points of $\mathcal{A}$ are periodic orbits of the Reeb vector field.  
\begin{defn} Let $(M, \xi)$ be a contact manifold with contact form $\alpha.$  The \textbf{action spectrum}, $$\sigma(\alpha) = \{r \in \bb{R} | r = \mathcal{A}(\gamma)\}$$ for $\gamma$ a periodic orbit of the Reeb vector field.
 
\end{defn}
\begin{defn} Let $T \in \sigma(\alpha).$  Let $$N_{T} = \{ p \in M | \phi_{p}^{T} = p\},$$ $$S_{T} = N_{T}/S^{1},$$ where $S^1$ acts on $M$ via the Reeb flow.  Then $S_{T}$ is called the \textbf{orbit space} for period $T$.
 \end{defn}
Here the orbit spaces are precisely the orbifold strata when $M$ is considered as an $S^1$ orbibundle.

For our Morse-Bott set-up we assume that our contact form is of Morse-Bott type, ie.  
\begin{defn}
A contact form is said to be of Morse-Bott type if
\begin{enumerate}
 \item[i.] The action spectrum:
$$\sigma(\alpha) := \{r \in \mathbb{R} : \mathcal{A}(\gamma) = r,\, for\, some\, periodic\, Reeb\, orbit\, \gamma.\}$$ is discrete.
\item[ii.] The sets $N_{T}$ are closed submanifolds of M, such that the rank of $d\alpha|_{N_{T}}$ is locally constant and $$T_{p}(N_{T}) = ker(d\phi_{T} - I).$$
\end{enumerate}
\end{defn}
\begin{remark}These conditions are the Morse-Bott analogues for the functional on the loop space of $M.$
\end{remark}

Notice that in the case of $S^1$ bundles this is always satisfied.  The key observation, as we soon shall see, is that we can relate $J$-holomorphic curves to Morse theory on the symplectic base.  Since we consider only quasi-regular contact manifolds here, we can always approximate the contact structure by one with a dense open set of periodic orbits of period 1, say, and a finite collection of strata of orbits of smaller period, each such stratum has even dimension and has codimension at least 2 (other than the dense set of regular points, of course).

\subsection{Moduli}
Before we describe the Morse-Bott chain complex we need to describe the moduli spaces of pseudoholomorphic curves with which we will be working.  So, as before let $(M, \xi)$ be a contact manifold, $V$ its symplectization, and $J$ an almost complex structure on $V$ compatible with the transverse symplectic structure on $M.$   The curves that we're interested in are $J$-holomorphic maps from punctured $S^2$'s into the the symplectization of our contact manifold.  Such curves are \emph{asymptotically cylindrical} over closed Reeb orbits.
\begin{defn}
Let $\Sigma$ be a Riemann surface with a set of punctures $Z = \{z_{1}, \ldots, z_{k}\}$  A map $$u = (a(s,t), f(s,t)): \Sigma \setminus Z \rightarrow V$$ is called \textbf{asymptotically cylindrical} over the set of Reeb orbits $\gamma_{1}, \ldots, \gamma_{k}$ if for each $z_{j} \in Z$ there are cylindrical coordinates centered at $z_{j}$ such that 
$$lim_{s \rightarrow \infty} f(s, t) = \gamma (T t)$$ 
and
$$lim_{s \rightarrow \infty} \frac{a(s, t)}{s} = T$$
\end{defn}    

Here are some precise statements from ~\cite{BEHWZ}:
\begin{prop}Suppose that $\alpha$ is of Morse,\footnote{Here we mean that for each periodic Reeb orbit, the time 1 return map has no eigenvalue equal to $1$.} or Morse-Bott type. Let $$u = (a, f ) : \bb{R}^{+} \times \bb{R}/\bb{Z} \rightarrow (\bb{R} \times M, J)$$ be a $J$-holomorphic curve of finite energy.\footnote{For the most general definitions of energy see ~\cite{BEHWZ}}  Suppose that the image of $u$ is unbounded in $\bb{R} \times M$. Then there exist a number $T \neq 0$ and a periodic orbit $\gamma$ of $R_{\alpha}$ of period $|T|$ such that

$$lim_{s \rightarrow \infty} f(s, t) = \gamma (T t)$$ 
and
$$lim_{s \rightarrow \infty} \frac{a(s, t)}{s} = T.$$

\end{prop}

This immediately implies 
\begin{prop}
Let $(S,j)$ be a closed Riemann surface and let $$Z=\{z_1,\ldots , z_k \}\subset S$$ be a set of punctures.  Every $J$-holomorphic curve $$F = (a,f):(S \setminus Z)\rightarrow \bb{R} \times M$$ of finite energy and without removable singularities is asymptotically cylindrical near each puncture $z_i$ over a periodic orbit $\gamma_i$ of $R_{\alpha}$.
\end{prop}

These propositions are extremely important to us because they show that it is reasonable to define gradient trajectories between Reeb orbits to be $J$-holomorphic curves in the symplectization.   
We have even more, ie., a Gromov compactness theorem, which says that although the moduli spaces are not necessarily compact, we can compactify them by adding stable curves of height $2$.  Even better, via a gluing construction ~\cite{Bourg1} we show that the boundary of the moduli space is equal to set of height $2$ curves.  

As above we consider the moduli space of generalized $J$-holomorphic curves from an $s$-times punctured Riemann surface into $V$ asymptotically cylindrical over the orbit spaces $S_1, \ldots, S_{s}$ representing the class $A$:
$$\mathcal{M}^{A}_{0,J}(S_{1}; \ldots, S_{s}, V).$$

In this notation the first orbit space is from the \emph{positive} puncture, all others are negative.  These moduli spaces are the analogues of the gradient trajectories of Morse theory.  We only count them when they come in zero dimensional families.  Thus we need a dimension formula for these spaces.  
\begin{prop}
 The virtual dimension for the moduli space of generalized genus $0$ $J$-holomorphic asymptotic over the orbit spaces $S_{T_{0}}, \ldots, S_{T_{1}},\ldots, S_{T_{s}}$ (with 1 positive, and $s$ negative punctures) representing $A$ is equal to
$$(n-3)(1-s)  + \mu(S_{T_{0}}) + \frac{1}{2}dim(S_{T_{0}}) - \sum_{i=0}^s (\mu(S_{T_{i}}) + \frac{1}{2} dim (S_{T_{i}})) + 2c_{1}(\xi, \Sigma),$$ where $\Sigma$ is a Riemann surface used to define the symplectic trivialization and homology class $A.$ 
\end{prop}

In cylindrical contact homology, since we only are keeping track of cylinders, we take $s=1$  and this formula reduces to $$ \mu(S_{T_{+}}) + \frac{1}{2}dim(S_{T_{+}}) + \mu(S_{T_{-}}) + \frac{1}{2} dim (S_{T^{-}}) + 2c_{1}(\xi, \Sigma).$$
Of course if $\xi$ has a regular structure this boils down to $$ \mu(S_{T^{+}}) - \mu(S_{T^{-}}) + 2n-2 + 2c_{1}(\xi, \Sigma).$$
For a proof of this formula see ~\cite{Bourg1}.  Bourgeois' proof is of interest as traditionally these kinds of results come from a spectral flow analysis.  Bourgeois, however, makes interesting use of the Riemann-Roch theorem.
 
We want to understand the structure of the moduli space since our Morse-like chain complex uses these curves to construct the differential.  The reader should be aware that the formula for the dimension of the moduli space above is really a \emph{virtual} dimension until some sort of transversality is achieved for some $\delbar_{J}$-operator. This formula is obtained via Fredholm analysis on the space  of $C^{\infty}$ maps from $S^2 \setminus \{z_1, z_2, \ldots, z_{j}\}$ into $V.$ The $\delbar_{J}$ turns out to be a Fredholm section of a certain infinite dimensional bundle over this space whose kernel is precisely the set of $J$-holomorphic curves.  The Fredholm index $\delbar_{J}$ is the dimension formula above.  The trouble is that a priori, we cannot rule out a non-zero cokernal, hence our dimension formula could be \emph{under counting} the relevant curves.  There have been many attempts at transversality proofs, and it seems as though the new \emph{polyfold} theory of Hofer, Zehnder and Wysocki should solve the problem.  There are also proofs using virtual cycle techniques, (cf. ~\cite{Bourg1}) however even here it seems that there may be potential gaps.  Therefore we show how, in some cases, we can justify the validity of our curve counts through more elementary geometric considerations.

In the cases that we are considering in this paper, the almost complex structure will be integrable, thus we can use algebro-geometric techniques to find conditions for regularity of $J$.\footnote{The word \emph{regularity} is over used.  Here we mean that for this $J$, the $\delbar_{J}$ operator is surjective, as a section in a suitable infinite dimensional vector bundle.}  

Now let us describe the relationship between moduli spaces of stable curves in a symplectic orbifold and the moduli space of curves into the symplectization of its Boothby-Wang manifold.  Notice that the symplectization $W$ is just the associated line (orbi)bundle to the principle $S^1$-(orbi)bundle\footnote{Of course, in the situations we are dealing with in this paper, $M$ is a manifold even if it is the total space of an orbibundle.}, $M$, with the zero section removed.   Given as many marked points as punctures we actually get a fibration, here curves upstairs are sections of $L$ with zeroes of order $k$ and poles of order $l$ once we fix the phase of a section we actually get unique curves.  This is described for the case of \emph{regular} contact structures in ~\cite{EGH}.  For $S^{1}$-bundles over $\bb{CP}^1$ with isolated cyclic singularities Rossi extended this result in ~\cite{PR}, we'll actually need an extension of this to higher dimension.  The point here is that we want to coordinate our curve counts upstairs with the ``Gromov-Witten'' curve count downstairs.  In the case where the base is an orbifold we must make sure that we can get an appropriate curve in the sense of Gromov-Witten theory on orbifolds.  It should be noted that in symplectizations all moduli spaces come with an $\bb{R}$-action by translation.  Whenever we talk about $0$-dimensional moduli spaces, we really mean that we are considering $1$-dimensional moduli spaces quotiented out by the $\bb{R}$-action giving $0$-dimensional manifolds\footnote{Actually these moduli spaces in general are branched orbifolds with corners, however in the cases that we consider in this paper they really are manifolds.}.

The following lemma comes from ~\cite{CR}.
\begin{lemma}
 Suppose that $u$ is a $J$-holomorphic curve into the symplectic orbifold $\mathcal{Z}$, then either $u$ is completely contained in the orbifold singular locus or it intersects it in only finitely many points.
\end{lemma}
We use this to prove:
\begin{lemma}
Let $u : \Sigma  \rightarrow \mathcal{Z}$ be a non-constant $J$-holomorphic map between a Riemann surface  and a symplectic orbifold. Then there is a unique orbifold structure on $\Sigma$ and a unique germ of a $C^{\infty}$-lift $\tilde{u}$ of $u$ such that $u$ is an orbicurve.
\end{lemma}
\begin{proof}
First let us assume that the marked points are all mapped into the singular locus, since otherwise the curve only intersects the singular locus in a finite number of points.  Now $u_{z_{i}}$ corresponds to a closed Reeb orbit of non-generic period, ie., a curve in $S_{T_{k}}$, say. Take an element from the moduli space of curves into $W$ asymptotically cylindrical over $S_{T_{k}}$ in some slot.  We need only to take a local uniformizing chart equivariant with respect to $\bb{Z}_{T_{k}}.$
\end{proof}

From this we actually get a fibration.
\begin{prop}
 There is a fibration $$pr :\mathcal{M}_{0, J}(S_{T_{1}}; S_{T_{2}}, \ldots , S_{T_{k}}) \rightarrow \mathcal{M}_{0, k}(a_{1}, \ldots, a_{k}).$$ 
\end{prop}

With this understanding of the moduli space, assuming now that $J$ is integrable we see that that the linearized Cauchy-Riemann operator is the $\delbar$ operator of the right Dolbeault complex on $\mathcal{Z}$. This leads to the following criterion for regularity of $\delbar_{J}$ at each $u$ adapted to the $S^1$ bundle case from ~\cite{McDSal2}.  First we note that $u: \bb{CP}^1 \rightarrow V$, we can look at $u^*T(V),$ whose characteristic classes look like those of $u^{*}\xi.$  Over $\bb{CP}^{1}$, $u^*\xi$ splits as a sum of line bundles:
$$u^*\xi = \bigoplus_{j} L_{j}.$$   
This splitting doesn't necessarily work for general orbifolds as ambient spaces (unless the sphere doesn't intersect the orbifold singular locus at all, but for toric orbifolds this splitting principle always works ~\cite{Gmn97}.  Of course, in the case of a regular contact form, this is just the classical Groethendieck splitting principle. 
\begin{thm}[Regularity criterion for genus $0$ moduli spaces when $J$ is integrable]
Suppose that $J$ is integrable. Suppose that $$\langle c_{1}(L_{j}), A \rangle \geq -2 + s -t$$ for every $A \in H_{2}(\mathcal{Z})$ which is represented by a 2-sphere.  Then the linearized Cauchy-Riemann operator is surjective and the genus $0$ moduli space of curves with $s$ positive punctures and $t$ negative punctures is a smooth manifold of dimension given by the Fredholm index.  In the case that $\mathcal{Z}$ is an orbifold, we require that $$c_{1}(L_{j}) \geq  \sum_{\alpha}(1-\frac{1}{m_{\alpha}})c_{1}(\mathcal{O}(D_{\alpha}))  -2-s-t,$$ where $D_{\alpha}$ are branch divisors.
\end{thm}
\begin{proof}
Let $z_{1}, \ldots, z_{s}, \ldots z_{s +t}$ be distinct points on $S^2.$  Consider the divisor $$D= k_{1}z_{1} + \ldots + k_{s}z_{s} - k_{s+1}z_{s+1} - \ldots - k_{s+t}z_{s+t}.$$  Then the Cauchy-Riemann operator is just the $\delbar$-operator of the Dolbeault complex for the line orbibundle $L_{j} \otimes \mathcal{O}(D).$
$$ \delbar : \Omega^{0}(\bb{CP}^1, L_{j} \otimes \mathcal{O}(D)) \rightarrow \Omega^{0,1}(\bb{CP}^1, L_{j} \otimes \mathcal{O}(D)).$$
The cokernel of $\delbar$ is just the $(0,1)$ cohomology of that complex.  But we have the following isomorphisms: 
$$H_{\delbar}^{0,1}(\bb{CP}^1, L_{j} \otimes \mathcal{O}(D)) \simeq H_{\delbar}^{1,0}(\bb{CP}^1, (L_{j} \otimes \mathcal{O}(D))^*)^* \simeq H_{\delbar}^{0,1}(\bb{CP}^1, (L_{j} \otimes \mathcal{O}(D))^* \otimes K).$$
For the last group to be $0$, we must have $$c_{1}(L_{j} \otimes \mathcal{O}(D))^* \otimes K) <0.$$  This happens whenever $$c_{1}(L_{j}) > -2 - deg(D).$$  
In the quasi-regular case we check this in the orbifold sense. So we consider the orbifold first Chern class :
$$c_{1}^{orb}= c_{1}(L_{j}) - \sum_{\alpha}(1-\frac{1}{m_{\alpha}})c_{1}(\mathcal{O}(D_{\alpha}))$$ where the sum at the end is non-zero only in the presence of branch divisors, $D_{\alpha}.$  Here we must bound this below by $-2 -deg(D).$  
\end{proof}
\begin{thm}
$$\langle c_{1}(L), A \rangle \geq -2 + s -t$$ whenever $L$ is a line (orbi)bundle obtained by the Boothby-Wang fibration whose total space is either of the following:
\begin{enumerate}
\item[i.] a homogeneous contact manifold
\item[ii.] a toric Fano contact manifold.
\end{enumerate}
\end{thm}
\begin{proof}
The proof of (i) is nearly the same Proposition 7.4.3 in ~\cite{McDSal2} with $u^{*}TM$ replaced with $u^*(\xi).$ Note that in case (i) we are always dealing with manifolds at each level rather than with orbifolds.
For (ii), for a splitting of $$u^*(\xi) = \bigoplus_j L_{j}$$ we get sections and positivity of Chern classes via the Fano condition.  Here the splitting takes place in the toric orbifold, then all the lines are pulled back to the sphere.  
\end{proof}   

\begin{cor}
 For homogeneous and toric contact manifolds the dimensions of all genus $0$ moduli spaces are given by the Fredholm index as predicted.
\end{cor}
We would like now to set up the Morse-Bott chain complex.  This was originally done in ~\cite{Bourg1} and discussed for circle bundles in ~\cite{EGH}.  We have already discussed some of the basic setup, now, much like the case in Morse theory we would like to relate the Morse-Bott case to the generic case.  The idea is to perturb our contact structure so it's periodic orbits are in $1$-$1$ correspondence with the critical points of some Morse-function.  In our case we would like to use our moment maps to get a perfect Morse or Morse-Bott function $f.$
So first for appropriate $\epsilon$ and our Morse or Morse-Bott function$f$ we take the new contact form
$$ \alpha_{f} = (1 + \epsilon f) \alpha.$$
Then critical points of $f$ correspond to periodic orbits of $\alpha_{f}.$
If $f$ is a perfect Morse function as in the toric case we take this as our contact form.  Else, suppose $f$ is Morse-Bott, then we choose Morse functions on each critical submanifold.  Note that in the case of a Hamiltonian action of a compact Lie group all such submanifolds have even index and even dimension.  By $G$-invariance of the moment map this restricts to all orbit spaces.  
In this case the periodic orbits are in $1$-$1$ correspondence with critical points of Morse functions on the critical submanifolds of $f.$
We want to relate the Conley-Zehnder indices of the generic form with those of the original.  Since we have a $1$-$1$ correspondence between critical points and orbits, we will think of the chain complex associated to $\alpha_{f}$ as critical points of $f.$  The index is the grading so we'll write the Conley-Zehnder index of the orbit corresponding to $p$, $\mu_{CZ}(\gamma_{p})$ as $|p|.$  In the Morse case we have
$$ |p| = \mu(S_{T_{k}}) -\frac{1}{2}dim(S_{T_{k}}) + ind_{p} f.$$\\*  
In the case of a Morse-Bott function, we just proceed as in ordinary Morse theory to get a new Morse function on each critical submanifold, but now in this formula we use the Morse-Bott indices. 
Now we can define the differential for Morse-Bott contact homology.
$$dp = \partial p + \sum_{q} n_{pq} q.$$
Where $\partial p$ is just the Morse-Smale-Witten boundary operator, $n_{pq}$ is similar to the coeffecient in the generic case, and $ind_{p}$ is the Morse index for critical points.  Bourgeois proves, in his thesis ~\cite{Bourg1}, that this homology computes contact homology.  For us, the particular form of the differential does not matter much since it will vanish for index reasons.  

\begin{thm}(Bourgeois)
 When the homology defined above exists it is isomorphic to the standard contact homology for non-degenerate contact forms. 
\end{thm}
\section{Morse-Bott Contact Homology in the Homogeneous and Toric Cases}
Now we apply the results from the previous section.  Let us first set some notation.  Suppose first that $(M, \xi)$ is compact, simply connected, and admits a strongly Hamiltonian action of a Lie group as discussed in the introduction which is of Reeb type.  Then we know that there is a quasi-regular contact form $\alpha$ for $(M, \xi)$ equivariant with respect to the action.
As above, let $S_{T_{k}}$ denote the stratum in $\mathcal{Z} = M/(S^1)$ corresponding to Reeb orbits of period $T_{k}.$  Let $\Gamma_j$ denote the local uniformizing group for the stratum $S_{T_{k}}.$  Recall that each stratum is a K\"{a}hler sub-orbifold of $\mathcal{Z}.$ 
In what follows assume that $H^{*}(\mathcal{Z}; \bb{C})$ is a truncated polynomial ring generated by elements in $H^{2}(\mathcal{Z}; \bb{C})$, ie., the Chern classes coming from the symplectic reduction defining $\mathcal{Z}$ as a symplectically reduced orbifold.  Let us write such a basis of $H^2(\mathcal{Z};\bb{C})$ as $\{c_{1}, \ldots, c_{k} \}.$  Now choose $1$ forms $\tilde{c_{j}}$ representing the $c_{j}$'s   Now we just consider circle bundles over $\mathcal{Z}$ by choosing connection $1$-forms $\alpha$ with curvature $$ d \alpha = \sum_{j}\pi^{*} w_{j} \tilde{c_{j}}.$$  Notice that for $\mathcal{Z}$ a toric orbifold, this construction yields all possible toric contact structures of Reeb type.   Note that we implicitly choose a symplectic form $\omega = \sum w_{i} \tilde{c_{i}}$ on $\mathcal{Z}$ during this process.  Then $$c_{1}(T(\mathcal{Z})) = \sum \tilde{w_{i}} \tilde{c_i},$$  where $\tilde{w_i}$ is obtained via the spectral sequence for the Boothby-Wang fibration.

\begin{remark}
In the case of contact reduction in $\bb{C}^n$ by a circle (where the action is of Reeb type) the coeffecients of $|z_{j}|^2$ in the (circle) moment map can be chosen to be the $\tilde{w_{j}}$'s. 
\end{remark}   

Now we choose elements of $H_{2}(\mathcal{Z}; \bb{Z})$, $A_{1}, \ldots, A_{n}$, with $$\langle \tilde{c_{i}},A_{i} \rangle =1.$$  This is possible because the cohomology is a truncated polynomial ring generated by the $c_{j},$ all elements having even degree.  Now let $$A= \sum_{j} A_{j}.$$ Then for any K\"{a}hler suborbifold $ i :S \hookrightarrow \mathcal{Z},$ $$\sum_{i}\langle i^{*}\tilde{c_{i}},A \rangle$$ is nonzero.  Thus we can also do this for each $S_{T_j}$ by pulling the Chern classes back along the inclusion maps, then choosing homology classes in each stratum as above in terms of $i_{j}^{*} \tilde{c_{i}},$ where $i_{j}: S_{T_{j}} \rightarrow \mathcal{Z}$ is the inclusion, and $\{c_{i} \}$ are the Chern classes generating $H^{*}(\mathcal(Z); \bb{C}).$  Call the corresponding homology class $A_{S_{T_j}}.$ The purpose here is to find a nice diagonally embedded sphere with which to make our calculations.  Now let's use this set-up to do some index calculations. 
First we must find suitable trivializations and capping disks for Reeb orbits.  The idea here is to find two trivializations for each Reeb orbit, then use the loop property of the Maslov index to calculate the index via integration of $c_{1}(T(\mathcal{Z}))$ over the sphere obtained by gluing the two disks (from the symplectic trivializations) along their boundaries.  The author first encountered this idea in ~\cite{Bourg1} and ~\cite{EGH}, however this was only for the regular\footnote{Regular in the sense of foliation theory.} case.  So let $\gamma_{S_{T_j}}$ be a Reeb orbit of period $T_{j}$, living, of course, in the stratum $S_{T_j}.$  We now pull back $\xi$ via the inclusion map over $S_{T_j},$ $i_{j}.$  For the first disk we just cap off a tubular neighborhood of the Reeb orbit given by the product framing for $M$.  In this framing the Maslov index is $0$, since the return map is always the constant path in $Sp(2n-2, \bb{R})$ given by the identity.  Now we need another disk to glue along the Reeb orbit to get a sphere.  In order to do this consider a holomorphic sphere, ie., a map $u:S^2 \rightarrow S_{T_j}$ passing through $p \in S_{T_{j}}$ such that $[u] =A_{S_{T_j}}.$  This is always possible since the moment map is invariant and since we assume $\mathcal{Z}$ is simply connected, the Hurewicz homomorphism is surjective.  Now consider a holomorphic (orbi)section of $L$ over our sphere with a zero of order equal to the multiplicity of $\gamma$ and no pole.  Such a section exists since we are talking about line (orbi)bundles over $\bb{CP}^1.$  With this set-up we prove:
\begin{lemma}
 Let $M$ be an $S^1$-bundle over a symplectic orbifold admitting a Hamiltonian action of a compact Lie group, such that its cohomology is generated by the Chern classes associated to the action. Then the Maslov index of a Reeb orbit in the stratum $S_{T_{j}}$ of multiplicity $m$ is equal to $$\frac{2m}{|\Gamma_{j}|}\int_{A_{S_{T_j}}} i^{*}c_{1}(T(S_{T_{j}})),$$  moreover this number is an integer.
\end{lemma}
\begin{proof}
By the loop property of the Maslov index,  the Maslov index of the Reeb orbit is twice the Maslov index of the path of change of coordinate maps between the two disks glued along $\gamma.$  Since the disk was obtained via an (orbi)section over a sphere representing $A_{S_{T_j}}$, we get $$\mu(\gamma) = 2 \langle c_{1}(\xi) , \sigma(u) \rangle =2 \langle c_{1}(T(\mathcal{Z}) , A_{S_{T_{j}}} \rangle.$$  This is exactly $c_{1}^{orb}(T(S_{T_{j}}))$ evaluated on $A$.  Therefore the index of an orbit of multiplicity $m$ is $$2m \int_{A_{S_{T_j}}} c_{1}^{orb}(T(S_{T_{j}})).$$  Now going back to the work of Satake ~\cite{Sat57} to compute the integral of an orbifold characteristic class over a homology class, we take intersections with all orbifold strata and divide out by the orders of the local uniformizing groups and sum:$$2k \int_{A_{S_{T_j}}} c_{1}^{orb}(T(S_{T_{j}}))=2m \sum_{j}\frac{1}{|\Gamma_{j}|}\int_{A_{S_{T_j}} \cap \Sigma_{j}} c_{1}(T(S_{T_{j}}))$$ where $\Gamma_{j}$ is a local uniformizing group in the orbifold stratum $\Sigma_{j}=S_{T_j}.$  Now, since each such spherical class is completely contained in $S_{T_{j}},$ we can just compute the integral $$\frac{2}{|\Gamma_{j}|}\int_{A_{S_{t_j}}} c_{1}(T(S_{T_{j}}))|_{S_{T_{j}}} = \frac{2}{|\Gamma_{j}|}\int_{A_{S_{T_j}}} i^{*}c_{1}(T(S_{T_{j}}))$$ for simple orbits, multiplying by $m$ for $m$-multiple orbits.  Note however that, although we may compute the integral on $\mathcal{Z}$, this integral is equal to one which takes place as the evaluation of an integral form on the contact manifold, hence we always get an integer.  
\end{proof}
\begin{remark}
The idea above is that $A_{S_{T_j}}$ is a ``sufficiently diagonal'' sphere in $S_{T_k}.$  This ensures that we pick up as much information as possible about the line bundle as possible during the integration.  One should also note that in general $c_{1}(\xi) \neq 0$ so this grading scheme for contact homology is computed with respect to a \emph{particular} choice of capping surface for each Reeb orbit.  When comparing contact manifolds which are $S^1$-orbibundles over the same base, care must be taken to make the same choices each time, so that the weights are realized via the Chern classes of each \emph{specific toric structure}.   
\end{remark}
\begin{remark}
The reader may wonder what role branch divisors play in the index calculation above.  This is encoded in summing over the strata and dividing by the orders of local uniformizers.
\end{remark}
We want to use these calculations to compute cylindrical contact homology, however this is not well defined unless we can exclude Reeb orbits of degree $0$, $1$, $-1$.  To ensure this we must assume that for all $k$ $$2 (\sum_i i^*c_{i}\tilde{w_{i}}) -\frac{1}{2} dim(S_{T_{k}})>0.$$ For this it is sufficient to assume that $$\sum_{i}\tilde{w_{i}} >1.$$  We take this as a standing assumption in the following.  

Now we notice that there are no rigid $J$-holomorpic cylinders other than the trivial ones.  This follows from a simple index computation and comparison with the dimension formula for the relevant moduli spaces.  This means that the contact homology is given ompletely by the Morse-Smale-Witten complex of the moment map with degree shifts given by the Maslov indices.  The discussion above yields theorem ~\ref{thm:main}].  We obtain the following corollaries.

\begin{cor}
 Let $(M, \xi)$ be a simply connected compact homogeneous contact manifold.  Then $CH_{*}(M)$ is generated by copies of $H_{*}(\mathcal{Z})$ with degree shifts given by $$2m \int_{A} c_{1}(T(\mathcal{Z})) = 2m \sum_{i} \tilde{w_{i}} -2$$.
\end{cor}
\begin{proof}In this case $M$ is an $S^{1}$-bundle over a generalized flag manifold, (recall that in this case there is a \emph{regular} contact $1$-form, $\alpha$ for $\xi$).  The cohomology of the base is a polynomial ring as per the discussion earlier, and all the relevant homology classes are spherical.  By the regularity theorem for integrable $J$ the dimension of the moduli space is the one predicted by the Fredholm index.  Moreover by the index calculation above, and the dimension formula for the moduli spaces, there are no rigid $J$-holomorphic curves connecting orbit spaces.  This contact homology is given compeltely in terms of the Morse-Smale-Witten differential, which vanishes since the moment map deterines a perfect Morse function, thus we get a generator for each critical point of the norm squared of the moment map in degree given by the Maslov indices as calculated in the previous discussion.
\end{proof}
\begin{cor} Let $(M, \xi)$ be a simply connected compact toric Fano contact manifold with a quasiregular contact form $\alpha$.  Then $CH_{*}(M)$ is generated by copies of $H_{*}(\mathcal{Z})$ with degree shifts given by the Maslov indices plus the dimension of the stratum containing the particular Reeb orbit as a point.  
If $\xi$ has a regular contact form $\alpha$ then the degree shifts are given by $$2m\sum_{j}\tilde{w_{j}}-2 ,$$ where the $\tilde{w_{j}}$ are defined as above.  
\end{cor}
\begin{proof}The Fano condition gives transversality of the $\delbar_{J}$-operator via the Dolbeault complex.  If we assume transversality we can drop the Fano assumption.  Again, our cohomology ring is a truncated polynomial ring generated by all possible Chern classes, with spherical second homology because of simple connectivity.  The indices are given by the even multiples of the sum of the weights.  Again there are no non-trivial $J$-curves.  So the homology is that given by the Morse-Smale-Witten complex (whose differential again vanishes by perfection of the Morse function) with the degree shifts given by the Maslov indices as calculated above. 
\end{proof}

\section{Examples} 
\subsection{Wang-Ziller Manifolds}
Now let's specialize to Wang-Ziller manifolds.  
These are toric manifolds either obtained from reduction in $\bb{C}^2 \times \bb{C}^2$ via the moment map $\mu(z,w) = k|w|^2 - l|z|^2.$  This manifold is also a homogeneous contact manifold.  This is how we get transversality.  Note that as a toric manifold, this manifold is non-Fano, so we really need to use the homogeneity to get transversaltiy via the Dolbeault complex.    
We can also see this manifold as a Boothby-Wang manifold.  Consider $\mathcal{Z}= \bb{C}\bb{P}^1 \times \bb{C}\bb{P}^1,$ and we take the standard symplectic form on each summand and multiply each piece by relatively prime integers $k$ and $l$.  
We take $P$ to be the circle bundle with a connection form $\alpha$ satisfying $d \alpha = \pi^{*}(k c_{1} + l c_{2}),$ where the $c_{j}$ are actually the generators of the second cohomology of each sphere.  Then $$c_{1}(\xi)=(2k - 2l)\beta,$$ for $\beta$ a generator of $H_2(S^2 \times S^3 ; \bb{Z}),$ and $$c_{1}(T\mathcal{Z}) = (2kc_{1} + 2lc_{2}),$$ here $\mathcal{Z}$ is topologically $\bb{CP}^1 \times \bb{CP}^1$ with the toric structure obtained by with symplectic form determined by $k,l.$        
$\mathcal{Z}$ admits a perfect Morse function, and the Maslov indices in this case for orbits of multiplicity $m$ are given by $ 4m(k+l) $.  Thus the grading of contact homology is given by
 $$|p| = 4m(k+l) - 2 + d$$
where $m\in \bb{Z} \setminus\{0\}$, and $d$ ranges over all possible degrees of homology classes in $\mathcal{Z},$ in this case $d=0,2,4.$b
This gives infinitely many distinct contact structure on $S^2 \times S^3$ since for each choice of relatively prime $k$ and $l$, we get generators of contact homology in minimal dimension $4(k+l) -2.$  Of course, for all pairs such that $k-l = c$ we get a single first Chern class for the contact bundle ~\cite{WZ}.  Choosing now all pairs with $k-l =c,$ we get infinitely many distinct contact structures in the same first Chern class.  In ~\cite{Ler1} Lerman showed that these contact structures are all pairwise non-equivalent as toric contact structures, but he asked whether or not they were pairwise contactomorphic.  This answers that question in the negative. Via the above construction we get contact structures $\xi_{k,l}$ on $S^{2} \times S^{3}.$
\begin{cor}
 Fix $c \in \bb{Z},$  choose $k,l$ such that $gcd(k,l) =1,$ and $k-l =c$ then the contact structures $\xi_{k,l}$ are pairwise non-contactomorphic all within the same first Chern class of $4$-plane distribution.  
\end{cor}

\begin{remark}
 This example suggests a Kunneth-type formula for the \textbf{join} ~\cite{BG07} construction for quasiregular contact manifolds provided each summand has suitable contact homology.  Suppose that $(\mathcal{Z}_1, \omega_{1})$ and $(\mathcal{Z}_2 , \omega_{2})$ are both simply connected symplectic orbifolds which are reduced spaces so that their cohomology rings are polynomials in the Chern classes.  Then we can build circle bundles over their product with curvature forms given as integer linear combinations of the $\omega_{j}.$  By choosing appropriate spheres ``diagonally'' embedded into the product we can evaluate the first Chern class of this bundle in order to get the Maslov indices as above.  Assuming transversality of the $\delbar_{J}$-operator this always computes contact homology.   
\end{remark}
\begin{remark}
 One could do a similar computation with Boothby-Wang spaces over $\bb{CP}^2 \# \overline{\bb{CP}^2}.$
\end{remark}

\subsection{Circle bundles over weighted projective spaces.}      
In these examples first off, note that weighted projective spaces are toric Fano, (even in the orbifold sense.)\footnote{These orbifolds generally have branch divisors.}  Note that the cohomology ring is then just the standard one, and we just need to find the right spherical classes.  Of course we just pullback $k$-multiples of the standard symplectic form to define our line bundles.  Now to compute contact homology of the bundle we integrate $c_{1}(T(\mathcal{Z}))$ over the class of a line.  The base admits a perfect Morse function, so all we need to do is keep track of the strata.  Integrating $i^{*}_{j} c_{1}(T(\mathcal{Z}))$ over spheres representing the K\"{a}hler class for each stratum.  So the grading of contact homology for an orbit of multiplicity $m$ will be $$2m \sum_{j}( \langle i_{j}^{*}c_{1}^{orb}(T\bb{CP}^n), [S^2_{j}]\rangle  + \frac{1}{2} dim S_{T_{j}})+ d+n-3=(2km\sum_{j}\frac{1}{|\Gamma_{j}|})+ \frac{1}{2}dimS_{T_{j}} + d+n-3,$$ where the class $[S^{2}_{j}]$ is the class of a line in each stratum and $d$ corresponds to the possible degree of a homology class on $\bb{CP}^{n}$, hence is an even number between $0$ and $2n$ and $S \in \bb{Z}^{+}.$  The dimension of the moduli space for genus $0$ and $1$ positive and $1$ negative puncture is then never $1.$  Notice that $c_{1}(\xi) = 0$ in this case.  So again we see that these contact manifolds are distinguished by the bundle and orbifold data. 
\begin{remark}
 One should be able to simplify the above formula when working with branch divisors.  We choose to stick with our earlier notation, in which any information about such branch divisors is encoded in the calculation.
\end{remark}

\newcommand{\etalchar}[1]{$^{#1}$}
\def\cprime{$'$}
\providecommand{\bysame}{\leavevmode\hbox to3em{\hrulefill}\thinspace}
\providecommand{\MR}{\relax\ifhmode\unskip\space\fi MR }
\providecommand{\MRhref}[2]{%
  \href{http://www.ams.org/mathscinet-getitem?mr=#1}{#2}
}
\providecommand{\href}[2]{#2}


\begin{thebibliography}{BEH{\etalchar{+}}03}

\bibitem[BE89]{BE}
Robert Baston and Michael Eastwood, \emph{The {P}enrose transform: Its
  interaction with representation theory}, Oxford University Press, Oxford,
  1989.

\bibitem[BEH{\etalchar{+}}03]{BEHWZ}
F.~Bourgeois, Y.~Eliashberg, H.~Hofer, K.~Wysocki, and E.~Zehnder,
  \emph{Compactness results in symplectic field theory}, Geom. Topol.
  \textbf{7} (2003), 799--888 (electronic). \MR{MR2026549 (2004m:53152)}

\bibitem[BG00a]{BG00}
Charles~P. Boyer and Krzysztof Galicki, \emph{A note on toric contact
  geometry}, J. Geom. Phys. \textbf{35} (2000), no.~4, 288--298. \MR{MR1780757
  (2001h:53124)}

\bibitem[BG00b]{BG00b}
Charles~P. Boyer and Krzysztof Galicki, \emph{On {S}asakian-{E}instein
  geometry}, Internat. J. Math. \textbf{11} (2000), no.~7, 873--909.

\bibitem[BG08]{BG}
\bysame, \emph{Sasakian geometry}, Oxford University Press, Oxford, 2008.

\bibitem[BGG82]{BGG}
I.~N. Bern{\v{s}}te{\u\i}n, I.~M. Gel{\cprime}fand, and S.~I. Gel{\cprime}fand,
  \emph{Schubert cells and the cohomology of the spaces ${G}/{P}$},
  Representation Theory, London Mathematical Socitey Lecture Note Series,
  vol.~69, Cambridge University Press, 1982.

\bibitem[BGO07]{BG07}
Charles~P. Boyer, Krzysztof Galicki, and Liviu Ornea, \emph{Constructions in
  {S}asakian geometry}, Math. Z. \textbf{257} (2007), no.~4, 907--924.
  \MR{MR2342558 (2008m:53103)}

\bibitem[Bor53]{Bor1}
A.~Borel, \emph{Sur la cohomologie des espaces fibr{\'{e}}s principaux et des
  espace homog{\`{e}}nes des groupes de lie compacts}, Ann. of Math.
  \textbf{57} (1953), 115--207.

\bibitem[Bou02]{Bourg1}
Frederic Bourgeois, \emph{A {M}orse-{B}ott approach to contact homology}, {PhD}
  thesis, Stanford University, 2002.

\bibitem[BW58]{BW58}
W.~M. Boothby and H.~C. Wang, \emph{On contact manifolds}, Ann. of Math. (2)
  \textbf{68} (1958), 721--734. \MR{MR0112160 (22 \#3015)}

\bibitem[CR02]{CR}
Weimin Chen and Yongbin Ruan, \emph{Orbifold {G}romov-{W}itten theory},
  Orbifolds in mathematics and physics ({M}adison, {WI}, 2001), Contemp. Math.,
  vol. 310, Amer. Math. Soc., Providence, RI, 2002, pp.~25--85. \MR{MR1950941
  (2004k:53145)}

\bibitem[EGH00]{EGH}
Y.~Eliashberg, A.~Givental, and H.~Hofer, \emph{Introduction to symplectic
  field theory}, Geom. Funct. Anal. (2000), no.~Special Volume, Part II,
  560--673, GAFA 2000 (Tel Aviv, 1999). \MR{MR1826267 (2002e:53136)}

\bibitem[Gir]{Gi1}
Emmanuel Giroux, \emph{Une structure de contact, m{\^{e}}me tendue est plus ou
  moins tordue}, Ann. Scient. e Ec. Norm. Sup. \textbf{27}.

\bibitem[GS99]{GlSt}
Victor Guillemin and Shlomo Sternberg, \emph{Supersymmetry and equivariant de
  {R}ham theory}, Springer-Verlag, Berlin, Heidelberg, New York, 1999.

\bibitem[Gui97]{Gmn97}
Victor Guillemin, \emph{Riemann-{R}och for toric orbifolds}, J. Differential
  Geom. \textbf{45} (1997), no.~1, 53--73. \MR{MR1443331 (98a:58075)}

\bibitem[Koe05]{OVK}
Otto~Van Koert, \emph{Open books for contact five-manifolds and applications of
  contact homology}, {PhD} thesis, University of K{\"{o}}ln, 2005.

\bibitem[Ler02]{Ler2}
Eugene Lerman, \emph{Contact toric manifolds}, Journal of Symplectic Geometry
  \textbf{1} (2002), no.~4, 785--828.

\bibitem[Ler03]{Ler1}
\bysame, \emph{Maximal tori in the contactomorphism groups of circle bundles
  over {H}irzebruch surfaces}, Mathematical Research Letters \textbf{10}
  (2003), 133--144.

\bibitem[LT97]{LerTol}
Eugene Lerman and Susan Tolman, \emph{Hamiltonian torus actions on symplectic
  orbifolds and toric varieties}, Trans. Amer. Math. Soc. (1997).

\bibitem[MS95]{McDSal}
Dusa McDuff and Dietmar Salamon, \emph{Introduction to symplectic topology},
  Oxford University Press, Oxford, 1995.

\bibitem[MS04]{McDSal2}
\bysame, \emph{{$J$}-holomorphic curves and symplectic topology}, American
  Mathematical Society, Providence, Rhode Island, 2004.

\bibitem[Ros]{PR}
Paolo Rossi, \emph{Gromov-{W}itten theory of orbicurves, the space of
  tri-polynomials and symplectic field theory of {S}eifert fibrations,
  arxiv:0808.2626}.

\bibitem[Sat57]{Sat57}
Ichir{\^o} Satake, \emph{The {G}auss-{B}onnet theorem for {$V$}-manifolds}, J.
  Math. Soc. Japan \textbf{9} (1957), 464--492. \MR{MR0095520 (20 \#2022)}

\bibitem[SR93]{SalRob}
Dietmar Salamon and Joel Robbins, \emph{The {M}aslov index for paths}, Topology
  \textbf{32} (1993), no.~4.

\bibitem[Ust99]{Ust}
Ilya Ustilovsky, \emph{Infinitely many contact structures on ${S}^{4m+1}$},
  International Mathematics Research Notices (1999), no.~14.

\bibitem[WZ90]{WZ}
McKenzie Wang and Wolfgang Ziller, \emph{Einstein metrics on principle torus
  bundles}, J. Differential Geometry \textbf{31} (1990), no.~1, 215--248.

\end{thebibliography}
\end{document}